\documentclass{amsart}
\usepackage{amsmath, amscd, amssymb, amsthm}
\usepackage{bbm}
\usepackage{latexsym}
\usepackage{amsfonts}
\usepackage{graphicx}
\usepackage[all,cmtip]{xy}
\usepackage[colorlinks,linkcolor=blue,breaklinks=blue,urlcolor=blue,citecolor=blue,anchorcolor=blue,pagebackref]{hyperref}%
\usepackage{geometry}
\newtheorem{theorem}{Theorem}
\newtheorem{lemma}{Lemma}
\newtheorem{corollary}[theorem]{Corollary}

\newtheorem{proposition}{Proposition}

\newtheorem*{definition}{Definition}

\renewcommand*\backref[1]{}
\renewcommand*\backrefalt[4]{ \ifcase #1 \or (cited on page #2) \else (cited on pages #2) \fi}

\newcommand{\be}{\begin{equation}}
\newcommand{\ee}{\end{equation}}
\newcommand{\bea}{\begin{eqnarray}}
\newcommand{\eea}{\end{eqnarray}}

\newcommand{\vs}{\vspace{0.5cm}}

\def\XXint#1#2#3{{\setbox0=\hbox{$#1{#2#3}{\int}$ }
\vcenter{\hbox{$#2#3$ }}\kern-.6\wd0}}

\begin{document}

\title[torsion-critical manifolds]{On a variational theorem of Gauduchon and torsion-critical manifolds}

\author{Dongmei Zhang}
\address{Dongmei Zhang. School of Mathematical Sciences, Chongqing Normal University, Chongqing 401331, China}
\email{{2250825921@qq.com}}\thanks{Zheng is partially supported by National Natural Science Foundations of China
with the grant No.12071050 and  12141101, Chongqing grant cstc2021ycjh-bgzxm0139, and is supported by the 111 Project D21024.}

\author{Fangyang Zheng}
\address{Fangyang Zheng. School of Mathematical Sciences, Chongqing Normal University, Chongqing 401331, China}
\email{20190045@cqnu.edu.cn; franciszheng@yahoo.com} \thanks{}

\subjclass[2010]{53C55 (primary), 53C05 (secondary)}
\keywords{Hermitian manifold; Chern connection; variation; Gauduchon's torison $1$-form; torsion-critical manifolds}

\begin{abstract}
In 1984, Gauduchon considered the functional of $L^2$-norm of his torsion $1$-form on a compact Hermitian manifold. He obtained in \cite{Gau84} the Euler-Lagrange equation for this functional, and showed that in dimension $2$ the critical metrics must be balanced (namely with vanishing torsion $1$-form). In this note we extend his result to higher dimensions, and show that critical metrics are balanced in all dimensions. We also consider the $L^2$-norm of the full Chern torsion, and show by examples that there are critical points of this functional that are not K\"ahler.
\end{abstract}

\maketitle

\tableofcontents

\markleft{Dongmei Zhang and Fangyang Zheng}
\markright{Goldberg manifolds}

\section{Introduction and statement of results}\label{intro}

Let $(M^n,g)$ be a compact Hermitian manifold of complex dimension $n$. Denote by $\omega$ its K\"ahler form. {\em Gauduchon's torsion $1$-form,} which is the trace of the torsion tensor $T$ of the Chern connection $\nabla$, is the global $(1,0)$-form $\eta$ on $M^n$ defined by $\partial (\omega^{n-1}) = -\eta \wedge \omega^{n-1}$. In the celebrated work \cite{Gau84}, Gauduchon considered the functional
\begin{equation*}
{\mathcal G} (g) = V^{\frac{1-n}{n}} \int_M |\eta|^2 dv, \ \ \ \ \ \ g\in {\mathcal H}_M,
\end{equation*}
where ${\mathcal H}_M$ is the set of all Hermitian metrics on the complex manifold $M^n$, $dv=\frac{\omega^n}{n!}$ the volume form, and $V=\int_Mdv$ the volume. Note that the factor in front of the integral is simply to make the functional independent of the scaling of metric (by constant multiples). Equivalently, one can drop this $V$ factor but restrict to the subset of all Hermitian metrics with unit volume. He obtained  the Euler-Lagrange equation of ${\mathcal G}$:
\begin{equation} \label{eq:Gau-crit}
\sqrt{-1} (\partial \overline{\eta} - \overline{\partial} \eta - \eta \wedge \overline{\eta}) = a\, \omega
\end{equation}
where $a \geq 0$ is a constant. This is formula (48bis) of \cite{Gau84}, where it was written in terms of the Lee form $\theta$, which is just $-(\eta +\overline{\eta})$. Here we rewrote it in the equivalent but more convenient complex form. He also proved \cite[Theorem III.4]{Gau84} that, when $n=2$, any critical point of ${\mathcal G}$ must be balanced, namely, with $\eta=0$.

In this note, we extend his theorem to general dimensions:

\begin{proposition} \label{prop1}
If a compact Hermitian manifold $(M^n,g)$ satisfies equation (\ref{eq:Gau-crit}), then $\eta=0$. That is, the only critical points of the Gauduchon functional $\,{\mathcal G}$  are balanced metrics.
\end{proposition}

Next, we mimic Gauduchon's functional and consider the $L^2$-norm of the Chern torsion tensor $T$:
\begin{equation} \label{eq:F}
{\mathcal F} (g) = V^{\frac{1-n}{n}} \int_M |T|^2 dv, \ \ \ \ \ \ g\in {\mathcal H}_M.
\end{equation}
Here $T$ is defined by $\,T(x,y)=\nabla_xy-\nabla_yx-[x,y]\,$ for any vector fields $x$, $y$ on $M^n$. For convenience, let us introduce the following terminology
\begin{definition}
A compact Hermitian manifold $(M^n,g)$ is said to be {\bf torsion-critical}, if $g$ is a critical point of ${\mathcal F}$.
\end{definition}

Since $T=0$ if and only if $g$ is K\"ahler, we see that all K\"ahler metrics are torsion-critical, and they are the absolute minimum of the functional ${\mathcal F}$. So the point is to understand torsion-critical metrics that are not K\"ahler. It turns out that unlike the case of Gauduchon functional, there do exist non-K\"ahler metrics that are torsion-critical for each $n\geq 3$.

In order the describe the equations obeyed by torsion-critical metrics, we need to fix some notations first. Given a Hermitian manifold $(M^n,g)$, let $\{ e_1, \ldots , e_n\}$ be a local unitary frame of type $(1,0)$  tangent vectors,  and let $\{ \varphi_1, \ldots , \varphi_n\}$ be the unitary coframe of $(1,0)$-forms dual to $e$. Under the frame $e$, the components of $T$ are:
$$ T(e_i, \overline{e}_k)=0, \ \ \ \ T(e_i, e_k) = \sum_r T^r_{ik} e_r , \ \ \ \  \ \ \ \ \ \ 1\leq i,k\leq n.$$
We have $\eta = \sum_i \eta_i \varphi_i$, $\ \eta_i=\sum_r T^r_{ri}$. So $\eta$ is the trace of $T$. Under any unitary local frame $e$ and for any $1\leq i,j\leq n$, let us denote by
\begin{equation}
A_{i\bar{j}} = \sum_{r,s}T^r_{is}\overline{T^r_{js}}, \ \ \ B_{i\bar{j}} = \sum_{r,s}T^j_{rs}\overline{T^i_{rs}}, \ \ \ \phi_{i}^j = \sum_r T^j_{ir} \overline{\eta}_r, \ \ \ \ \xi_{i}^j = \sum_r T^j_{ir,\,\bar{r}}
\end{equation}
where the index after comma stands for covariant derivative with respect to Chern connection. Define
\begin{eqnarray*}
&& \sigma_1= \sqrt{-1}\sum_{i,j} A_{i\bar{j}} \varphi_i \wedge \overline{\varphi}_j, \ \ \ \ \ \ \sigma_2= \sqrt{-1}\sum_{i,j} B_{i\bar{j}} \varphi_i \wedge \overline{\varphi}_j, \\
&& \phi = \sqrt{-1}\sum_{i,j} \phi_{i}^j \varphi_i \wedge \overline{\varphi}_j, \ \ \ \ \ \ \ \ \xi = \sqrt{-1}\sum_{i,j} \xi_{i}^j \varphi_i \wedge \overline{\varphi}_j.
\end{eqnarray*}
Clearly, $\sigma_1$, $\sigma_2$, $\phi$, $\xi$ are independent of the choice of local unitary frames, hence are globally defined $(1,1)$-forms on $M^n$, with $\sigma_1\geq 0$, $\sigma_2\geq 0$. Let us denote by
$$ |T|^2:= \sum_{i,j,k} |T^j_{ik}|^2$$
and call it the {\em square norm} of the torsion tensor $T$. Then the trace of $\sigma_1$ or $\sigma_2$ with respect to the K\"ahler form $\omega = \sqrt{-1}\sum_{i} \varphi_i\wedge \overline{\varphi}_i$ is equal to $|T|^2$, so $\sigma_1=0$ (or $\sigma_2=0$) if and only if $g$ is K\"ahler. We have the following

\begin{proposition}  \label{prop2}
A compact Hermitian manifold $(M^n,g)$ is torsion-critical if and only if
\begin{equation} \label{eq:tor-crit}
2\sigma_1 - \sigma_2 +2(\phi + \bar{\phi}) - 2 (\xi + \bar{\xi}) = (|T|^2-\frac{n-1}{n}b)\, \omega, \ \ \ \ \ \ b=V^{-1}\int_M|T|^2dv.
\end{equation}
That is, $g$ is a critical point of the functional ${\mathcal F}$ if and only if it satisfies equation $(\ref{eq:tor-crit})$.
\end{proposition}

Note that when $n=2$, $\eta$ and $T$ carry the same amount of information, to be more precise one has $|T|^2=2|\eta|^2$ when $n=2$, so ${\mathcal F}$ and ${\mathcal G}$ are essentially the same. So by the aforementioned theorem of Gauduchon we know that

\begin{corollary}[\cite{Gau84}]
For $n=2$, any torsion-critical metric must be balanced, hence K\"ahler.
\end{corollary}

In general dimensions, the trace of $\phi$ with respect to $\omega$ is $|\eta|^2=\sum_r |\eta_r|^2$, while the trace of $\xi$ is the global function $\chi = \sum_r \eta_{r,\,\bar{r}}$. Since
$$ \sqrt{-1}\partial \overline{\partial} \omega^{n-1} = ( |\eta|^2-\chi ) \omega^n, $$
we know that $\chi$ is a real-valued function on $M^n$, and $|\eta|^2-\chi =0$ identically if and only if the metric is Gauduchon. By taking trace with respect to $\omega$ on both sides of (\ref{eq:tor-crit}), we get the following result which is due to Angella, Istrati, Otiman, and Tardini, \cite[Proposition 17]{AIOT}:

\begin{corollary}[\cite{AIOT}]
Denote by $[g]$ the conformal class of $g$, namely, the set of all Hermitian metrics on $M^n$ conformal to $g$. Then $g$ is a critical point of ${\mathcal F}|_{[g]}$ if and only if
\begin{equation}\label{eq:trace}
4\big(|\eta|^2-\chi \big) = (n-1) \big(|T|^2-b\big),
\end{equation}
where $b$ is the average value of $|T|^2$. In particular, the metric will be Gauduchon if and only if $|T|^2$ is a constant.
\end{corollary}

Note that the Euler-Lagrange equation of ${\mathcal F}|_{[g]}$ is just the trace of that of ${\mathcal F}$. In particular, all  torsion-critical metrics satisfy (\ref{eq:trace}). To see that equation (\ref{eq:trace}) is the same as (18) of \cite{AIOT}, we note that under any unitary coframe $\varphi$, it holds that $\partial \omega = \sqrt{-1}\sum_{i,j,k} T^j_{ik} \varphi_i\wedge\varphi_k \wedge\overline{\varphi}_j$. Hence $|d\omega|^2 =2|\partial \omega|^2 = |T|^2$, and our  functional ${\mathcal F}$ here is the same as the functional ${\mathcal A}$ studied in \cite{AIOT}.  For the Lee form $\theta = -(\eta+\overline{\eta})$, one has $d^{\ast}\theta = 2(\chi -|\eta|^2)$.

In the paper \cite{AIOT}, the authors studied the variational problem for a number of interesting geometric functionals, with a special emphasis on their restriction on conformal classes. The papers \cite{ACS} and \cite{ACS1} also give excellent discussions on related problems.

We are particularly interested in the set ${\mathcal H}^{tc}_n$ of $n$-dimensional non-K\"ahler, torsion-critical manifolds, and especially its subset of all balanced ones, for $n\geq 3$. At present, we have very limited understanding about this class, except knowing that it is non-empty on one hand and highly restrictive on the other hand,  in the sense that many familiar types of special Hermitian metrics are not in it (unless K\"ahler).

Denote by $\nabla^s$ the Strominger connection \cite{Strominger} of $(M^n,g)$. It is also known as Bismut connection \cite{Bismut} in many literature. A Hermitian manifold is called {\em Strominger torsion parallel,} or STP in short, if $\nabla^sT=0$. We have the following

\begin{proposition} \label{prop3}
Suppose $(M^n,g)$ is a compact STP manifold. Then it will be torsion-critical if and only it is balanced and with $\sigma_2=c\,\omega$ for some constant $c$.
\end{proposition}

\begin{corollary}
For any $k\geq 3$, the complex Lie group $SO(k,{\mathbb C})$, equipped with a compatible metric, is balanced, non-K\"ahler, SPT, and with $\sigma_2$ equal to a constant multiple of $\omega$. So any compact quotient of it is a non-K\"ahler, balanced, torsion-critical manifold.
\end{corollary}

Since the product of torsion-critical manifolds are torsional-critical, so by taking product, we know that the set ${\mathcal H}^{tc}_n$ (and actually its subset of balanced ones) is non-empty for all $n\geq 3$. This settles the non-emptiness of the class. To illustrate its restrictiveness, we make the following observation:

\begin{proposition} \label{prop4}
Let $(M^n,g)$ be a compact Hermitian manifold belonging to one of the following
\begin{enumerate}
\item locally conformally K\"ahler, or
\item Strominger K\"ahler-like, or
\item complex nilmanifold with nilpotent $J$ in the sense of \cite{CFGU}, or
\item all the complex nilmanifolds or Calabi-Yau type complex solvmanifolds in dimension $3$, listed by \cite[Table 1 and 2]{AOUV}
\end{enumerate}
Then $g$ cannot be torsion-critical unless it is already K\"ahler.
\end{proposition}

Since the class seems to be so restrictive, it would be very attempting to try to classify it, at least in low dimensions or for some special types of Hermitian metrics. For instance, one could try to

{\em Classify all balanced, non-K\"ahler, torsion-critical manifolds in dimension $3$.}

{\em Classify all torsion-critical manifolds that are Chern flat \cite{Boothby}.}

{\em Classify all torsion-critical manifolds that are pluriclosed, namely, $\partial \overline{\partial}\omega =0$.}

Similar to balanced manifolds, pluriclosed manifolds (also known as {\em strong K\"ahler with torsion}, or SKT for short in many literature) are an important class of special Hermitian metrics that are widely studied. We refer the readers to the excellent paper \cite{FinoTomassini} for more discussion.

The variational method is a classical approach in geometry and analysis, with a long and rich history. Various types of special non-K\"ahler metrics are also widely studied in Hermitian geometry. As a limited sampler, we refer the readers to the following papers and the references therein for more discussions:  \cite{AI}, \cite{AU}, \cite{EFV}, \cite{FinoVezzoni}, \cite{Fu}, \cite{Gau97}, \cite{Goldberg}, \cite{KYZ}, \cite{Schoen}, \cite{STW}, \cite{Tosatti}, \cite{TW}, \cite{Tseng-Yau}, \cite{WYZ}, \cite{YZ}, \cite{YZ1},  and \cite{Zheng}.

This short article is organized as follows. In \S 2, we will give a proof to Proposition \ref{prop1}. In \S 3, we will deduce the Euler-Lagrange equation, namely, Proposition \ref{prop2}. In \S 3, we will prove Propositions \ref{prop3} and \ref{prop4}.

\vspace{0.3cm}

\section{Proof of Proposition \ref{prop1}}

Let $(M^n,g)$ be a compact Hermitian manifold of complex dimension $n\geq 2$.  Let us first verify that equation (\ref{eq:Gau-crit}) is indeed the Euler-Lagrange equation of ${\mathcal G}$, namely, it is the same as formula (48bis) of \cite{Gau84}.

Recall that the {\em Lee form} $\theta$ is the $1$-form on $M^n$ defined by $d\omega^{n-1}=\theta \wedge \omega^{n-1}$, where $\omega$ is the K\"ahler form. So clearly, one has
 $$ \theta = - (\eta + \overline{\eta}), \ \ \ \ \mbox{hence} \ \ \  \ J\theta = \sqrt{-1}(\overline{\eta}-\eta ).$$
Therefore,
$$ (dJ\theta)^{1,1} = \sqrt{-1}(\partial \overline{\eta} - \overline{\partial}\eta ), \ \ \ \ \ \ \theta \wedge J\theta = -2 \sqrt{-1}\eta \wedge \overline{\eta}. $$
So by Gauduchon's formula (48bis) in \cite{Gau84}, we get equation (\ref{eq:Gau-crit}).

Now assume that $g$ is a critical point of the Gauduchon functional ${\mathcal G}$. So we have equation (\ref{eq:Gau-crit}), which is in complex form thus being more convenient. Gauduchon's proof of his Theorem III.4 in \cite{Gau84} for the $n=2$ case is an integration-by-part argument, using Stokes' theorem repeatedly. Below we give a brief account of his proof for readers convenience.

\begin{proof}[{\bf Gauduchon's proof of Proposition \ref{prop1} in the $n=2$ case}]

Let $(M^2,g)$ be a compact Hermitian surface with $g$ satisfying equation (\ref{eq:Gau-crit}), where the constant $a$ is equal to $\frac{1}{2V}\int_M|\eta|^2dv$.  Rewrite the equation as
$$ L:=\sqrt{-1}(\partial \overline{\eta} - \overline{\partial}\eta ) = a\,\omega + \sqrt{-1}\eta\wedge \overline{\eta} :=R.$$
For both sides, take the wedge product of the $(1,1)$-form with itself, we get
\begin{eqnarray*}
 R\wedge R &= &a^2 \omega^2 + 2a \sqrt{-1} \eta\wedge \overline{\eta} \wedge \omega = (a^2+a|\eta|^2)\omega^2,\ \ \ \ \mbox{while} \\
L\wedge L & = & - \partial \overline{\eta} \,\partial \overline{\eta} - \overline{\partial}\eta \, \overline{\partial}\eta + 2 \partial \overline{\eta}\, \overline{\partial}\eta \\
&=& -\partial (\overline{\eta} \,\partial \overline{\eta} ) - \overline{\partial} (\eta \, \overline{\partial}\eta) + 2 \{ \partial(\overline{\eta} \,\partial \overline{\eta}) + \overline{\partial} (\overline{\eta} \,\partial \eta ) - \partial \eta \, \overline{\partial \eta} \}
\end{eqnarray*}
Therefore $\,\int\! L\wedge L = - \int \partial \eta \, \overline{\partial \eta} \,\leq 0$. On the other hand, $\int\! R\wedge R= 3a^2V\geq 0$, so it forces $a=0$ hence $\eta =0$. This completes the proof of Gauduchon's theorem.
\end{proof}

Note that when $n\geq 3$, one has to wedge the above $(2,2)$-forms with $\omega^{n-2}$, which messes up the exact forms, and further integration-by-part will involve the differential of $\omega$, and the integral of the left hand side is no longer clearly non-positive. So one has to seek an alternative argument.

\begin{proof}[{\bf Proof of Proposition \ref{prop1}}]
Let $(M^n,g)$ be a compact Hermitian manifold with $n\geq 2$ satisfying equation (\ref{eq:Gau-crit}). The constant $a$ is given by $\frac{1}{nV}\int |\eta|^2dv$. Our goal is to show that $a=0$. Assume the contrary, namely, $a>0$. Let us as before move the third term of the left hand side to the right, and get
$$ \sqrt{-1}(\partial \overline{\eta} - \overline{\partial}\eta ) = a\,\omega + \sqrt{-1}\eta\wedge \overline{\eta} := \omega_0 .$$
Since $a>0$, $\omega_0\geq a\omega $ is the K\"ahler form of another Hermitian metric on $M^n$. Let $\alpha = -\sqrt{-1}\partial \eta$ and $\psi = \sqrt{-1}(\overline{\eta}-\eta)$. Then we have
$ d\psi = \omega_0 + \alpha + \overline{\alpha}$. Since $\alpha \overline{\alpha }\geq 0$, we have
$$ (d\psi )^n = \omega_0^n + C_n^2C_2^1\omega_0^{n-2}\alpha \overline{\alpha} + C_n^4C_4^2\omega_0^{n-4}(\alpha \overline{\alpha})^2 + \cdots \geq \omega_0^n , $$
hence $\int_M (d\psi)^n>0$. But that is a contradiction since $(d\psi )^n=d \big( \psi \wedge (d\psi)^{n-1}\big)$ is exact. This means that the assumption $a>0$ can not occur, therefore $a$ must be zero, thus $\eta=0$. This completes the proof of Proposition \ref{prop1}.
\end{proof}

In other words, under the assumption that $a>0$, the above construction leads to an exact symplectic form $d\psi$ (namely the Liouville type symplectic manifold), which cannot occur on closed manifold, and this argument works in all dimensions.

\vspace{0.3cm}

\section{Torsion-critical manifolds}

In this section we will derive the Euler-Lagrange equation of the functional ${\mathcal F}$ given by (\ref{eq:F}) and prove Proposition \ref{prop2}.

Let $(M^n,g)$ be a compact Hermitian manifold. Denote by $\nabla$ the Chern connection and $T$ its torsion tensor. Extend $g=\langle , \rangle$ bilinearly over ${\mathbb C}$. Under a local holomorphic coordinate system $z=(z_1,\ldots , z_n)$, consider the natural frame $\{ \partial_1 , \ldots , \partial_n\}$ where $ \partial_i= \frac{\partial}{\partial z_i}$, $1\leq i\leq n$. Write $g_{i\bar{j}}=\langle \partial_i , \overline{\partial}_j\rangle$ for the entries of the $n\times n$ matrix of $g$, which we will also denote by $g$, and denote its inverse matrix by $g^{-1}=(g^{\bar{j}i})$. Under the natural frame, $\nabla$ has components
\begin{equation*}
\nabla_{\partial_k} \partial_i = \sum_{j} \Gamma^j_{ik} \partial_j  = \sum_{j} \big( \sum_{\ell} \partial_k(g_{i\bar{\ell}})\,g^{\bar{\ell}j} \big) \partial_j.
\end{equation*}
We have $T(\partial_i,\overline{\partial}_j)=0$, $\ T(\partial_i,\partial_k)=\sum_j T^j_{ik}\partial_j$ with
\begin{equation} \label{eq:T}
T^j_{ik} = \sum_{\ell} \big(\partial_i(g_{k\bar{\ell}}) - \partial_k(g_{i\bar{\ell}})\big) g^{\bar{\ell}j}.
\end{equation}

Let $\omega = \sqrt{-1}\sum_{i,j} g_{i\bar{j}} dz_i \wedge d\bar{z}_j$ be the K\"ahler form of $g$. Gauduchon's torsion $1$-form $\eta$ is given by $\eta = \sum_i \eta_i dz_i$ where
\begin{equation} \label{eq:eta}
\eta_i = \sum_k T^k_{ki} = \sum_{k,\ell} \big( \partial_k(g_{i\bar{\ell}}) - \partial_i(g_{k\bar{\ell}}) \big) g^{\bar{\ell}k} = \sum_k \big( \Gamma^k_{ik} -\Gamma^k_{ki} \big) .
\end{equation}

Under a local coordinate system, we have $\, |\eta|^2=\sum_{i,j}\eta_i \overline{\eta_j} g^{\bar{j}i}\,$, $\ \chi = \sum_{i,j}\eta_{i,\bar{j}}g^{\bar{j}i}\,$, and
\begin{eqnarray*}
  && \sigma_1 = \sqrt{-1} \sum_{i,j} A_{i\bar{j}} dz_i \wedge d\overline{z}_j, \ \ \ \ \ \ \ A_{i\bar{j}}=\sum_{r,s,p,q}T^r_{is} \overline{T^p_{jq}} g_{r\bar{p}} g^{\bar{q}s}, \\
 &&  \sigma_2 = \sqrt{-1} \sum_{i,j} B_{i\bar{j}} dz_i \wedge d\overline{z}_j, \ \ \ \ \ \ \ B_{i\bar{j}}=\sum_{r,s,p,q,k,\ell}T^{\ell}_{rs} \overline{T^k_{pq}} g^{\bar{p}r} g^{\bar{q}s}g_{i\bar{k}} g_{\ell \bar{j}}\,, \\
 && \phi = \sqrt{-1} \sum_{i,j} \phi_{i}^{\ell}g_{\ell \bar{j}} \,dz_i \wedge d\overline{z}_j, \ \ \ \ \ \ \phi_{i}^j=\sum_{r,s}T^j_{ir}\overline{\eta_s}  g^{\bar{s}r} ,\\
 &&  \xi = \sqrt{-1} \sum_{i,j} \xi_{i}^{\ell} g_{\ell \bar{j}} \, dz_i \wedge d\overline{z}_j, \ \ \ \ \ \ \ \ \xi_{i}^j=\sum_{r,s}T^j_{ir,\bar{s}}  g^{\bar{s}r} .
\end{eqnarray*}

Now suppose that $X$ is a type $(1,0)$-vector field on $M^n$. In a coordinate neighborhood, write $X=\sum^i X^i \partial_i$. The {\em divergence} of $X$ is the global function on $M^n$ defined by $div(X) = \sum_i X^i_{\,,i}$, where the index after comma stands for covariant derivative with respect to $\nabla$. Denote by $\lrcorner$ the contraction on differential forms by vector fields. As is well known, we have
\begin{equation*}
d (X \lrcorner \,  \omega^n) = \big( div (X) - \eta (X) \big) \omega^n,
\end{equation*}
Integrating over $M^n$, we get the divergence theorem on Hermitian manifolds:

\begin{lemma} \label{lemma1}
Let $(M^n,g)$ be a compact Hermitian manifold, and $X$ a type $(1,0)$ vector field on $M^n$. Then it holds that
\begin{equation*}
\int_M div (X) \omega^n  =  \int_M \eta (X) \omega^n . \label{divergence}
\end{equation*}
\end{lemma}

Next, we consider tensor fields on $M^n$. Using the Hermitian metric $g$,  tensor fields can be transformed from covariant type to contra-variant type and vice versa. So we just need to consider tensor fields of `pure type', namely, with only holomorphic part but without anti-holomorphic part. Let us denote by ${\mathcal T}^q_{\,p}$ the vector space of all pure type $(p,q)$ tensor fields on $M$. In a local coordinate neighborhood, $C\in {\mathcal T}^q_{\,p}$ has the expression
$$ C = \sum C_{I}^J \,dz_{i_1}\otimes \cdots \otimes dz_{i_p} \otimes \partial_{j_1}\otimes \cdots \otimes \partial_{j_q}.$$
Here $\partial_i=\frac{\partial}{\partial z_i}$,  and we write for convenience $I=(i_1, \ldots , i_p)$, $J=(j_1, \ldots , j_q)$. Each $i_k$ or $j_k$ is summed from $1$ to $n$. The Hermitian metric $g$ on $M^n$ naturally induces a Hermitian inner product $(\,,\,)$ on ${\mathcal T}_{\,p}^q$ by
$$ (C,D) = \int_M \langle C, \overline{D} \rangle \,dv ,$$ where
$$ \langle C, \overline{D} \rangle = \sum  C_I^J\,\overline{D_K^L} \, g^{\bar{k}_1i_1} \cdots g^{\bar{k}_pi_p} \, g_{j_1\bar{\ell}_1} \cdots g_{j_q\bar{\ell}_q} ,$$
and all indices are summed from $1$ to $n$. For $C\in {\mathcal T}^q_{\,p}$, its $(1,0)$-covariant derivative $\nabla 'C \in {\mathcal T}^q_{p\!+\!1}$ is defined by
$$ (\nabla'C)_{i_1\cdots i_{p\!+\!1}}^{j_1 \cdots  j_q} = (\nabla_{i_{p\!+\!1}}C)_{i_1 \cdots i_{p}}^{j_1 \cdots  j_q} = C_{i_1 \cdots i_{p}, \, i_{p\!+\!1} }^{j_1 \cdots  j_q}. $$
Applying the divergence theorem, we get the following well-known integration-by-part formula, and we include a brief proof here for readers convenience.

\begin{lemma} \label{lemma2}
On a compact Hermitian manifold $(M^n,g)$, for any $C\in {\mathcal T}^q_{\,p}$ and $D\in {\mathcal T}^q_{p\!+\!1}$, it holds that
$$ ( \nabla ' C , D) = ( C, D^1\!-\!D^0), $$
 where $D^1$, $D^0\in {\mathcal T}^q_{\,p}$ are given respectively by
 $$ (D^1)_I^J = \sum_{i,j} D_{Ii}^J \overline{\eta}_jg^{\bar{j}i}, \ \ \ \ (D^0)_I^J = \sum_{i,j} D^J_{Ii,\,\bar{j}} g^{\bar{j}i} .$$
\end{lemma}

\begin{proof}
Let $X$ be the type $(1,0)$ vector field on $M^n$ given in any local coordinate neighborhood by $X=\sum_i X^i\partial_i$ where
$$ X^i = \sum C_{i_1\cdots i_p}^{j_1 \cdots j_q } \,\overline{  D_{k_1\cdots k_{p\!+\!1}}^{\ell_1 \cdots \ell_q }  } \, g^{\bar{k}_1i_1} \cdots g^{\bar{k}_pi_p} g^{\bar{k}_{p\!+\!1}i}\, g_{j_1\bar{\ell}_1} \cdots g_{j_q\bar{\ell}_q}. $$
Here and below we used Einstein's convention that repeated indices are summed up from $1$ to $n$. From this, we get
 $$ \eta (X) = \langle C, \overline{D^1}\rangle , \ \ \ \ \ div(X) = X^i_{\,,i} = \langle \nabla 'C, \overline{D}\rangle + \langle C, \overline{D^0} \rangle .$$
 Combine the two identities and use the divergence theorem (Lemma \ref{lemma1}), we get the integration by part formula, so Lemma \ref{lemma2} is proved.
\end{proof}

Now we are ready to derive the Euler-Lagrange equation for the $L^2$-norm of the Chern torsion. Let $(M^n,g)$ be a compact Hermitian manifold, and let $h$ be a Hermitian-symmetric covariant tensor of type $(1,1)$. Using the metric $g$ to lift the anti-holomorphic part, $h^j_i= \sum_{\ell}h_{i\bar{\ell}}g^{\bar{\ell}j}$, we may view $h$ as a tensor in ${\mathcal T}_{\,1}^1$ as well. For small real values $t\in (-\varepsilon ,\varepsilon )$, consider Hermitian metric $g(t)=g+th$ on $M^n$, with Chern torsion $ T_{ik}^j(t)$. For convenience, let us denote by $\overset{\,\circ}{P}$ the quantity $\frac{d}{dt}P|_{t=0}$ for any $P$ depending on $t$. We have
$$ \overset{\circ}{g(t)} = h, \ \ \ \overset{\circ}{\big(g(t)^{-1}\big)} = - g^{-1}hg^{-1}, \ \ \ \overset{\circ}{\big(dv(t)\big)} = \overset{\circ} {\big( \log \det g(t) \big)} \,dv = \langle h,\bar{g}\rangle \,dv .$$
By (\ref{eq:T}), we get
\begin{equation*} \label{eq:Tdot}
 \overset{\circ}{T^j_{ik}} = \big( \partial_i(h_{k\bar{\ell}}) - \partial_k(h_{i\bar{\ell}}) \big) g^{\bar{\ell}j} - T^r_{ik}h_{r\bar{q}}g^{\bar{q}j} = \big( h_{k\bar{\ell},\,i}-h_{i\bar{\ell},\,k} \big) g^{\bar{\ell}j},
\end{equation*}
where indices after comma stand for covariant derivatives with respect to $\nabla$. Since
$$ |T|^2 = T^j_{ik} \overline{T^b_{ac}}\, g^{\bar{a}i} g^{\bar{c}k}g_{j\bar{b}}\,, $$
we get
\begin{equation*} \label{eq:Tsquaredot}
 \overset{\circ}{|T|^2} = 2\mbox{Re} \langle \overset{\circ}{T}  , \overline{T}\rangle + \langle h, \overline{B-2A}\rangle  = -4\mbox{Re} \langle \nabla'h  , \overline{T}\rangle + \langle h, \overline{B-2A}\rangle
\end{equation*}
Now we can apply Lemma \ref{lemma2}, and notice that for the torsion tensor $D=T$, the corresponding tensors $(D^1)_i^j=\phi_i^j$ and $(D^0)_i^j=\xi_i^j$, so we get the following
\begin{equation}
\int_M \overset{\circ}{|T|^2} dv = \int_M \langle h, \overline{P}\rangle dv, \ \ \ \ \ \mbox{where} \ \ \ P = B-2A -2(\phi + \phi^{\ast}) + 2(\xi + \xi^{\ast}).
\end{equation}
Apply this to the functional ${\mathcal F}=V^{\frac{1-n}{n}}\int |T|^2dv$, we get $\overset{\circ}{{\mathcal F}}= \int_M \langle h, \overline{Q}\rangle dv$, where
\begin{equation}
Q=P+(|T|^2+\frac{1-n}{n}b)g, \ \ \ \ \ \ \ b=\frac{1}{V}\int_M |T|^2dv.
\end{equation}
So  $g$ will be a critical point of ${\mathcal F}$ if and only if $Q=0$, namely, $-P=(|T|^2+\frac{1-n}{n}b)g$. Writing it equivalently in terms of $(1,1)$-forms, this means
$$ 2\sigma_1 - \sigma_2 + 2(\phi + \overline{\phi}) - 2(\xi + \overline{\xi}) = (|T|^2+\frac{1-n}{n}b)\omega. $$
Therefore we have completed the proof of Proposition \ref{prop2}.

\vspace{0.3cm}

\section{Existence and  non-existence of torsion-critical metrics}

K\"ahler metrics are certainly torsion-critical, namely, are critical points  of the $L^2$-norm ${\mathcal F}$ of Chern torsion, and by Gauduchon's theorem, when $n=2$ there are no other critical points. When $n\geq 3$, however, there are examples of non-K\"ahler metrics that are torsion-critical. By Proposition \ref{prop2},  product of torsion-critical metrics are still torsion-critical, so we just need to find a non-K\"ahler torsion-critical metric in dimension $3$, then such examples will exist in all dimensions $n\geq 3$.

Consider a compact Chern flat manifold $(M^3,g)$ whose universal cover is the complex Lie group
$G=SO(3,{\mathbb C})$, consisting of all $3\times 3$ complex matrices $X$ satisfying $^t\!X X=I$, equipped with a metric compatible with the complex structure. $(M^3,g)$ is a non-K\"ahler, compact Chern flat manifolds. Such manifolds are always balanced and with Chern parallel torsion (namely, $\nabla T=0$), so we have $\eta=0$, $\phi=0$, and $\xi =0$.

To compute $\sigma_1$ and $\sigma_2$, let us take the standard left-invariant coframe $\varphi$ on $G$, which are left-invariant $(1,0)$-forms satisfying
$$ d\varphi_1=\varphi_2\wedge \varphi_3, \ \ \ d\varphi_2=\varphi_3\wedge \varphi_1, \ \ \ d\varphi_3=\varphi_1\wedge \varphi_2.$$
Using $\varphi$ as unitary coframe, the metric has Chern torsion components
$$ T^1_{23}=T^2_{31}=T^3_{12}=-1,$$
while other components are zero. From these structure constants, we get
$$ A_{i\bar{j}}=B_{i\bar{j}}=2\delta_{ij}, \ \ \ \ \forall \ 1\leq i,j\leq 3. $$
So $\sigma_1=\sigma_2=2\omega$. Also, $|T|^2=b=6$ is a constant, so equation (\ref{eq:tor-crit}) is satisfied, and $(M^3,g)$ is torsion-critical.

Note that the same conclusion holds for  compact quotients of $SO(k,{\mathbb C})$ for any $k\geq 3$: all such manifolds are non-K\"ahler, balanced, and torsion-critical.

The example was actually found by restricting our attention to Strominger parallel manifolds, or STP in short, which means Hermitian metrics satisfying $\nabla^sT=0$. Here $T$ is the Chern torsion and $\nabla^s$ is the Strominger (or Bismut) connection. We have the following:

\begin{lemma} \label{lemma3}
Let $(M^n,g)$ be a STP manifold. Then under any unitary frame $e$ it holds
\begin{eqnarray}
 T^j_{ik,\ell} & = &  T^j_{rk} T^r_{i\ell } + T^j_{ir} T^r_{k\ell } -  T^r_{ik} T^j_{r\ell }, \label{eq:Tcomma}\\
  T^j_{ik,\bar{\ell}} & = & -T^j_{rk}\overline{T^i_{r\ell }} - T^j_{ir}\overline{T^k_{r\ell }}+ T^r_{ik}\overline{T^r_{j\ell }}, \label{eq:Tcommabar}\\
  0 \ \ & = & T^j_{rk} T^r_{i\ell } + T^j_{ir} T^r_{k\ell } -  T^r_{ik} T^j_{r\ell } \label{eq:Teta}
\end{eqnarray}
for any any $1\leq i,j,k,\ell \leq n$, where $r$ is summed from $1$ to $n$, and index after comma stand for covariant derivatives with respect to the Chern connection $\nabla$.
\end{lemma}

\begin{proof}
Let $e$ be a local unitary frame, with dual coframe $\varphi$. As is well-known, the difference between $\nabla^s$ and $\nabla$ are given by
\begin{equation*}
(\nabla^s-\nabla )e_i = \sum_{j=1}^n \big( \sum_{r=1}^n\{ T^j_{ir} \varphi_r - \overline{T^i_{jr}} \overline{\varphi}_r  \}  \big)e_j.
\end{equation*}
See for example \cite{ZZ} or  \cite{YZZ}, and note that the $T^j_{ik}$ there are half of ours here. So if $\nabla^sT=0$, then by the above formula we immediately get formula (\ref{eq:Tcomma}) and (\ref{eq:Tcommabar}). To prove the last identity (\ref{eq:Teta}), apply formula (17) in \cite[Lemma 2]{ZZ}, where $R$ stands for the Riemannian curvature tensor. If we take the sum of that identity by cyclicly permuting the three indices without bar, the curvature part disappear by the first Bianchi identity, hence we get a formula involving only the derivative of $T$ and the quadratic terms in $T$. Combining this identity with (\ref{eq:Tcomma}), we know that both sides of (\ref{eq:Tcomma}) must be zero. This completes the proof of the lemma.
\end{proof}

Now we are ready to prove Proposition \ref{prop3} stated in the introduction part.

\begin{proof}[{\bf Proof of Proposition \ref{prop3}}]
Let $(M^n,g)$ be a STP manifold, namely, it satisfies $\nabla^sT=0$. In (\ref{eq:Tcommabar}) above, take $k=\ell$ and sum it up from $1$ to $n$, we get
$$ \xi^j_i = - B_{i\bar{j}} + \phi^j_i + A_{i\bar{j}}.$$
 Therefore $\phi - \xi = B - A$ for any STP manifold, and the left hand side of equation (\ref{eq:tor-crit}) becomes $\sigma_2$. This shows that any SPT metric will be torsion-critical if and only if $\sigma_2 = c\, \omega$ for some constant $c$.

 It remains to show that any SPT torsion-critical metric must be balanced. Assume that it is not, then we can fix a point in $M$ and a unitary frame $e$ such that $\eta_1=\cdots =\eta_{n-1}=0$ but $\eta_n\neq 0$ at the given point. On the other hand, by letting $j=\ell$ in (\ref{eq:Teta}) and sum it up, we know that any SPT manifold always satisfies
 $$ \sum_r \eta_r T^r_{ik}=0$$
 for any $1\leq i,k\leq n$. With the above choice of unitary frame $e$, we have $T^n_{ik}=0$ for any $i,k$. This means that $B_{n\overline{n}}=0$. In other words, the $B$ tensor of a non-balanced SPT manifold can never be positive definite. So when $g$ is assumed to be torsion-critical, then $B_{n\overline{n}}=0$ leads to $c=\frac{1}{n}|T|^2=0$, hence $T=0$ and the metric is K\"ahler, contradicting with the assumption that $g$ is not balanced. This completes the proof of Proposition \ref{prop3}.
\end{proof}

Finally, let us prove the non-existence result, Proposition \ref{prop4}, which searches for torsion-critical metrics amongst several special classes of Hermitian metrics.

\begin{proof}[{\bf Proof of Proposition \ref{prop4}}]
Let $(M^n,g)$ be a torsion-critical manifold. We want to show that $g$ must be K\"ahler if it belongs to one of the special classes listed in the proposition. First assume that $g$ is locally conformally K\"ahler. As is well-known, the torsion tensor of such a metric is determined by its torsion $1$-form, namely, under any unitary frame $e$ we have
$$ T^j_{ik} = \frac{1}{n-1}\big( \delta_{ij}\eta_k - \delta_{kj}\eta_i \big) $$
for any $1\leq i,j,k\leq n$. From this, we compute
\begin{eqnarray*}
&& A_{i\bar{j}} \, = \, \frac{1}{(n-1)^2} \big( \delta_{ij} |\eta|^2 + (n-2)\eta_i \overline{\eta}_j\big) , \ \ \ \ \ B_{i\bar{j}} \, = \, \frac{2}{(n-1)^2} \big( \delta_{ij} |\eta|^2 - \eta_i \overline{\eta}_j\big), \\
&& \phi_i^j \, = \, \frac{1}{n-1} \big( \delta_{ij} |\eta|^2 -\eta_i \overline{\eta}_j\big), \ \ \ \ \ \ \ \xi_i^j \, = \, \frac{1}{n-1} \big( \delta_{ij} \chi -\eta_{i,\,\bar{j}} \big), \ \ \ \ \  |T|^2\,  = \,\frac{2}{n-1}  |\eta|^2 .
\end{eqnarray*}
Since $g$ is assumed to be torsion-critical, by plugging the above into equation (\ref{eq:tor-crit}) we get
$$ \eta_{i,\bar{j}} + \overline{\eta_{j,\bar{i}}} - \eta_i\overline{\eta}_j = \frac{n-1}{2n}b \,\delta_{ij}, $$
or equivalently, equation (\ref{eq:Gau-crit}). So by Proposition \ref{prop1}, we know that $\eta=0$, hence $T=0$.

Next, let us assume that $g$ is Strominger K\"ahler-like, namely, the curvature tensor of $\nabla^s$ obeys all the K\"ahler symmetries. By \cite{ZZ}, we know that $\nabla^sT=0$, so $g$ is SPT. By Proposition \ref{prop3}, we know that $g$ must be balanced. By \cite{ZZ} again, balanced Strominger K\"ahler-like metrics are always K\"ahler, so we have proved that $g$ must be K\"ahler in this case.

Now let us assume that $(M^n,g)$ is a complex nilmanifold with nilpotent $J$ in the sense of \cite{CFGU}. By the beautiful result of Salamon \cite{Salamon}, there exists unitary left-invariant coframe $\varphi$ such that
\begin{equation}
C^j_{ik}=D_{jk}^i =0 \ \ \ \ \ \mbox{unless} \ j> i,k. \label{eq:Jnil}
\end{equation}
where $C^j_{ik}$ and $D^j_{ik}$ are structural constants of the nilpotent Lie algebra determined by
$$ d\varphi_j =  - \frac{1}{2} \sum_{i,k=1}^n C^j_{ik} \varphi_i \wedge \varphi_k - \sum_{i,k=1}^n \overline{D^i_{jk}} \, \varphi_i \wedge \overline{\varphi}_k .
$$
By \cite{VYZ} or \cite{ZZ1}, we have
\begin{eqnarray*}
T_{ik}^j & = & -C_{ik}^j  - D_{ik}^j + D_{ki}^j   \\
   T^j_{ik,\overline{\ell }}  & = &  \sum_{r=1}^n \big( T^j_{rk}\overline{\Gamma^i_{r\ell}} + T^j_{ir}  \overline{\Gamma^k_{r\ell}} - T^r_{ik} \overline{\Gamma^r_{j\ell}} \big),
\end{eqnarray*}
where $\Gamma_{ik}^j  =  D^j_{ik}$ is the coefficients of the Chern connection $\nabla$. Note that our $T^j_{ik}$ equals twice of that in \cite{VYZ} or \cite{ZZ1}. From the above, we get $\eta_i = \sum_s D^s_{is}$, and
$$ \xi_i^j = \sum_{r,s} \{ T^j_{rs} \overline{D^i_{rs}} - T^r_{is} \overline{D^r_{is}} \} + \phi_i^j. $$
Using the conditions (\ref{eq:Jnil}), we get
$$ \phi_n^n - \xi_n^n = - \sum_{r,s} |D^r_{ns}|^2, \ \ \ \ A_{n\bar{n}} = \sum_{r,s} |D^r_{ns}|^2, \ \ \  \ B_{n\bar{n}} = \sum_{r,s} |C^n_{rs}|^2. $$
Plug them into equation (\ref{eq:tor-crit}), we get
$$ -\sum_{r,s} \{ 2|D^r_{ns}|^2+|C^n_{rs}|^2 \} = \frac{1}{n}|T|^2 . $$
Hence $|T|^2=0$ and $g$ is K\"ahler.

Finally, let $(M^3,g)$ be any manifold listed in Table 1 or Table 2 of \cite{AOUV}, which are all the complex nilmanifolds and Calabi-Yau type complex solvmanifolds in dimension $3$. For the nilpotent ones in Table 1, we only need to check the type (Niii) since others are with nilpotent $J$. By direct verification, one could see that $g$ cannot be torsion-critical, unless it is already K\"ahler. We will omit the details here since it is a straight-forward computation. This completes the proof of Proposition \ref{prop4}.
\end{proof}

\vs

\noindent\textbf{Acknowledgments.} The second named author would like to thank Haojie Chen, Ping Li, Lei Ni, Xiaolan Nie, Kai Tang, Hongwei Xu, Bo Yang, Xiaokui Yang, and Quanting Zhao for their interests and/or helpful discussions.

\vs


\begin{thebibliography}{99}

\bibitem {AI} B. Alexandrov and S. Ivanov,  \emph{Vanishing theorems on Hermitian manifolds,}
Diff. Geom. Appl. {\bf 14} (2001),  251-265.

\bibitem  {ACS}  D. Angella, S. Calamai, C. Spotti, \emph{On the Chern-Yamabe problem,} Math. Res. Lett. {\bf 24} (2017),
no.3, 645-677.

\bibitem  {ACS1}  D. Angella, S. Calamai, C. Spotti, \emph{Remarks on Chern-Einstein Hermitian metrics,} arXiv: 1901.04309v3

\bibitem  {AIOT}  D. Angella, N. Istrati, A. Otiman, N. Tardini, \emph{Variational Problems in Conformal Geometry,}
J. Geom. Anal. {\bf  31} (2021), 3230-3251.


\bibitem  {AOUV}  D. Angella, A. Otal, L. Ugarte, R. Villacampa,  \emph{On Gauduchon connections with K\"ahler-like curvature,} arXiv:1809.02632v2.

\bibitem  {AU}  D. Angella and L. Ugarte,   \emph{Locally conformal Hermitian metrics on complex non-K\"ahler manifolds,} Mediterr. J. Math. {\bf 13} (2016) 2105-2145.

\bibitem {Bismut} J.-M. Bismut, \emph{A local index theorem for non-K\"ahler manifolds,}  Math. Ann. {\bf 284} (1989), no. 4, 681-699.

\bibitem {Boothby} W. Boothby, \emph{Hermitian manifolds with zero curvature,} Michigan Math. J. {\bf 5} (1958), no. 2, 229-233.


\bibitem{CFGU} L. Cordero, M. Fern\'{a}ndez, A. Gray, L. Ugarte, \emph{Compact nilmanifolds with nilpotent complex structures: Dolbeault cohomology}, Trans. Amer. Math. Soc. {\bf 352} (2000), no. 12, 5405-5433.

\bibitem{EFV} N. Enrietti, A. Fino and L. Vezzoni, \emph{Tamed symplectic forms and strong K\"ahler with torison metrics,} J. Symplectic Geom. {\bf 10} (2012), no.2, 203-223.


\bibitem {FinoTomassini} A. Fino and A. Tomassini, \emph{A survey on strong KT structures,} Bull. Math. Soc. Sci. Math. Roumanie, Tome 52 (100) No. 2, 2009, 99-116.

\bibitem {FinoVezzoni}  A. Fino, and L. Vezzoni, \emph{On the existence of balanced and SKT metrics on nilmanifolds,} Proc. Amer. Math. Soc., {\bf 144} (2016), no. 6, 2455-2459.

\bibitem {Fu} J-X Fu,  {\em On non-K\"ahler Calabi-Yau threefolds with balanced metrics.} Proceedings of the International
Congress of Mathematicians. Volume II, 705-716, Hindustan Book Agency, New Delhi, 2010.

\bibitem {Gau84} P. Gauduchon, \emph{La $1$-forme de torsion d'une
vari\'et\'e hermitienne compacte.} Math. Ann. {\bf 267} (1984), no. 4,
495-518.

\bibitem {Gau97} P. Gauduchon, \emph{Hermitian connnections and Dirac operators,}
Boll. Un. Mat. It. {\bf 11-B} (1997) Suppl. Fasc., 257-288.


\bibitem {Goldberg} S.I. Goldberg, \emph{ Tensorfields and curvature in Hermitian manifolds with torsion,} Ann. Math. {\bf 63} (1956), 64-76.


\bibitem {KYZ} G. Khan, B. Yang, and F. Zheng, \emph{The set of all orthogonal complex strutures on the flat $6$-torus,} Adv. Math. {\bf 319} (2017), 451-471.

\bibitem {Salamon} S. Salamon, \emph{Complex structures on nilpotent Lie algebras,} J. Pure Appl. Algebra {\bf 157} (2001), 311-333.


\bibitem {Schoen}  R. Schoen, \emph{ Variational theory for the total scalar curvature functional for Riemannian metrics and related topics,} Topics in Calculus of Variations, pp 120-154, Lecture Notes in Mathematics book series, volume 1365, Springer, Berlin, 1989.

\bibitem {Strominger} A. Strominger, \emph{Superstrings with Torsion,}
Nuclear Phys. B {\bf 274} (1986), 253-284.

\bibitem {STW} G. Sz\'ekelyhidi, V. Tosatti and B. Weinkove, \emph {Gauduchon metrics with prescribed volume form,}  Acta Math. {\bf 219} (2017), no. 1, 181-211.

\bibitem{Tosatti} V. Tosatti, \emph{Non-K\"ahler Calabi-Yau manifolds,} in Analysis, complex geometry, and
mathematical physics: in honor of Duong H. Phong, 261-277, Contemp. Math., 644, Amer. Math. Soc., Providence, RI, 2015

\bibitem{TW} V. Tosatti, B. Weinkove, \emph{The complex Monge-Amp`ere equation on compact Hermitian manifolds,}
J. Amer. Math. Soc. {\bf 23} (2010), no.4, 1187-1195.

\bibitem  {Tseng-Yau} L.-S. Tseng and S.-T. Yau, \emph{Non-K\"ahler Calabi-Yau
manifolds.} String-Math 2011, 241-254, Proc. Sympos. Pure Math.,
\textbf{85}, Amer. Math. Soc., Providence, RI, 2012.


\bibitem {VYZ} L. Vezzoni, B. Yang, and F. Zheng, \emph{ Lie groups with flat Gauduchon connections,} Math. Zeit. \textbf{293} (2019), Issue 1-2, 597-608.


\bibitem {WYZ} Q. Wang, B. Yang, and F. Zheng, \emph{On Bismut flat manifolds,} Trans. Amer.Math.Soc., {\bf 373} (2020), 5747-5772.

\bibitem {YZ} B. Yang and F. Zheng, \emph{On curvature tensors of Hermitian manifolds,} Comm. Anal. Geom. {\bf 26} (2018), no. 5, 1193-1220.

\bibitem {YZ1} B. Yang and F. Zheng, \emph{On compact Hermitian manifolds with flat Gauduchon conmnections,} Acta Math. Sinica (English Series). {\bf 34} (2018), 1259-1268.



\bibitem  {YZZ} S.-T. Yau, Q. Zhao, and F. Zheng, \emph{On Strominger K\"ahler-like manifolds with degenerate torsion,} arXiv:1908.05322v2, to appear in Trans. Amer. Math. Soc.

\bibitem {ZZ} Q. Zhao and F. Zheng, \emph{Strominger connection and pluriclosed metrics,} arXiv:1904.06604v3.

\bibitem {ZZ1} Q. Zhao and F. Zheng, \emph{Complex nilmanifolds and K\"ahler-like connections,} J. Geom. Phys. {\bf 146} (2019).

\bibitem {Zheng} F. Zheng, \emph{Some recent progress in non-K\"ahler geometry,} Sci. China Math., {\bf 62} (2019), no.11, 2423-2434.


\end{thebibliography}
\end{document}